\documentclass[12pt]{amsart}

\usepackage[pagewise]{lineno}

\usepackage[all]{xy}
\usepackage{xspace}
\usepackage{graphicx}
\usepackage{longtable}
\usepackage{mathtools}
\usepackage{latexsym}
\usepackage{fullpage,amsfonts,amsmath,amscd,amssymb}
\input{xy}
\xyoption{all}
\xyoption{curve}

\newcommand{\bA}{{\mathbb A}}
\newcommand{\bB}{{\mathbb B}}

\newcommand{\bD}{{\mathbb D}}
\newcommand{\bE}{{\mathbb E}}

\newcommand{\bG}{{\mathbb G}}
\newcommand{\bK}{{\mathbb K}}

\newcommand{\bS}{{\mathbb S}}
\newcommand{\bZ}{{\mathbb Z}}

\newcommand{\fg}{{\mathfrak g}}
\newcommand{\fl}{{\mathfrak l}}

\newcommand{\GL}{{{\mbox{\rm GL}}}}

\newcommand{\Aut}{{{\mbox{\rm Aut}}}}
\newcommand{\Ad}{{{\mbox{\rm Ad}}}}
\newcommand{\Lie}{{{\mbox{\rm Lie}}}}
 
\newcommand{\dom}{{{\mbox{\rm dom}}}}
\newcommand{\Hom}{{{\mbox{\rm Hom}}}}
\newcommand{\PGL}{{{\mbox{\rm PGL}}}} 

\newcommand{\Coker}{{{\mbox{\rm Coker}}}}

\newcommand{\Sta}{{{\mbox{\rm Stab}}}}

\newcommand{\Obs}{{{\mbox{\rm Obs}}}}
\newcommand{\Inf}{{{\mbox{\rm Inf}}}}
\renewcommand{\inf}{{{\mbox{\rm inf}}}}
\newcommand{\Res}{{{\mbox{\rm Res}}}}

\newcommand{\LH}{{[LH]^\theta{\mbox{\rm mo}}}} 
\newcommand{\End}{{{\mbox{\rm End}}}}

\renewcommand{\1}{_{(1)}}

\newtheorem{theorem}{Theorem}
\newtheorem{prop}[theorem]{Proposition}
\newtheorem{lemma}[theorem]{Lemma}

\newtheorem{cor}[theorem]{Corollary}

\begin{document}
\title[Integration of Modules]{Integration of Modules I: Stability}
\author{Dmitriy Rumynin}
\email{D.Rumynin@warwick.ac.uk}
\address{Department of Mathematics, University of Warwick, Coventry, CV4 7AL, UK
  \newline
\hspace*{0.31cm}  Associated member of Laboratory of Algebraic Geometry, National
Research University Higher School of Economics, Russia}
\thanks{The research was partially supported by the Russian Academic Excellence Project `5--100' and by the Max Planck Society.}
\author{Matthew Westaway}
\email{M.P.Westaway@warwick.ac.uk}
\address{Department of Mathematics, University of Warwick, Coventry, CV4 7AL, UK}
\date{September 20, 2018}
  
\begin{abstract}
  We explore the integration of
  representations
  from a Lie algebra to its algebraic group in positive characteristic.
  An integrable module is stable under the twists by group elements.
  Our aim is to investigate cohomological obstructions for passing from
  stability to an algebraic group action.
As an application, we prove integrability of bricks
for a semisimple algebraic group. 
\end{abstract}

\maketitle

Over a field of positive characteristic, an algebraic group $G$ acts on its Lie algebra $\fg$ and the restricted enveloping algebra $U_1(\fg)$
by automorphisms. This yields twists: an element $x\in G$
twists a $\fg$-module $(V,\theta)$
into $(V,\theta)^x \coloneqq (V,\theta\circ \Ad(x))$.
A $\fg$-module is {\em $G$-stable} if it is isomorphic to all its twists.
A $\fg$-module coming from a $G$-module is necessarily $G$-stable
but the converse is not true.
An important question in the modular representation theory of Lie algebras and algebraic groups is to determine for which modules the converse is true. We investigate this question in this paper.

Our method is subtly different from the known approach. 
Not only Cline and Donkin \cite{Cl, Don2} but also Parshall and Scott in their modern
exposition \cite{PS1} pursue a certain unipotent extension $G^\ast$
of the group $G$ that acts on a $G$-stable $\fg$-module $(V,\theta)$.
We, instead,  contemplate projective actions of $G$ on $(V,\theta)$.
In particular, we completely avoid the theory of Schreier Systems. 

Our approach instead has similarities to the work of Dade and Th\'{e}venaz \cite{Dade, Thev} on a related question for abstract groups. They study whether a $G$-stable representation $(V,\theta)$ of a normal subgroup $L$ of an abstract group $G$ can be extended to a representation of the entire group $G$. They show that when the automorphism group $\Aut_L(V)$ of $V$ is abelian the extension is controlled by 
$H^2(G/L, \Aut_L(V))$. 
Furthermore, the uniqueness of such extensions is controlled by 
$H^1(G/L, \Aut_L(V))$. 

By introducing the terminology of $(L,H)$-morphs - a type of function which is partway to being a homomorphism - we are able to reinterpret the results of Dade-Th\'{e}venaz in a more general context (Theorem \ref{exact_seq}). When we apply these results to the question of module extensions we repeat Corollary 1.8 and Proposition 2.1 in \cite{Thev}, however our formalism allows us to generalise this to the case where $\Aut_L(V)$ is only soluble rather than abelian via an inductive process (Theorem \ref{branching}).

The other key difference between our results and those of Dade-Th\'{e}venaz is that we work with a slightly different relative cochain complex, denoted $(C^\bullet(G,L;A),d)$, while they work with the more standard complex $(C^\bullet(G/L;A),d)$. Whilst the cohomology of these complexes differ from the second cohomology group on, Theorem \ref{exact_seq} in fact works in either case. However, in order to apply similar methods to the case of algebraic groups, the study of this new cochain complex becomes necessary. These considerations are explained in more detail in Sections \ref{s2.5} and \ref{s3.5}.

Now we reveal the detailed content of the present paper,
emphasising the main results. In Chapter~\ref{s2},
we devise all the machinery to discuss $G$-stable modules
in the setting of abstract groups:
a group $G$, its normal subgroup $L$ and a $G$-stable $L$-module $(V,\theta)$.
We introduce weak $(L,H)$-morphs and the relative cochain complex
$C^\bullet (G,L;A)$
in Section~\ref{s2.2}, where $A$ is an abelian group with a $G$-action.
They feature in a key exact sequence (see Theorem~\ref{exact_seq}) that
controls both uniqueness and existence
of $G$-actions for a large class of $G$-stable $L$-modules.

The main result of this chapter is Theorem~\ref{branching},
a somewhat algorithmic result pinpointing completely
uniqueness and existence of a $G$-module structure on
a $G$-stable $L$-module.
Notice that it has been established by
Xanthopoulos
that $H^1 (G/L; A)$ controls uniqueness \cite{Xan}.
Since $H^1 (G/L; A)= H^1 (G,L; A)$, our results about uniqueness
are known.
However, $H^2 (G/L; A)\neq H^2 (G,L; A)$ (and the latter controls existence),
hence our results on existence are new, even in the setting
of abstract groups.
Our approach is useful because it fuses uniqueness and existence
into a single process controlled by the relative cohomology.

In Chapter~\ref{s3}
we extend our Chapter~\ref{s2} results from abstract groups to
algebraic groups.
We face some technical challenges.
An important case for applications is when $L$ is a Frobenius kernel of $G$.
Hence, we must assume that $L$ is a closed subgroup scheme, not just
a closed algebraic subgroup.
The second challenge is poles:
we need to distinguish
rational and algebraic cohomology,
since we encounter rational cocycles $\mu : G^n\rightarrow A$
that are not necessarily algebraic.
We deal with technicalities in Sections~\ref{s3.1} and \ref{s3.2}.

In Section~\ref{s3.3}
we exhibit
a key exact sequence for rational cohomology
(Theorem~\ref{alg_exact_seq} -- an analogue of Theorem~\ref{exact_seq}).
Again,  this sequence 
controls both uniqueness and existence
of $G$-actions.
Immediately we put it to good use:
a $G$-stable $\fg$-brick (a module with trivial endomorphisms)
is a $G$-module (Theorem~\ref{brick}). 

A greater generality than 
$\fg$-bricks 
is $\fg$-modules with a soluble group
of automorphisms. These are our assumptions in Section~\ref{s3.4}.
Our main result in this section is Theorem~\ref{alg_branching},
an analogue of Theorem~\ref{branching} for algebraic groups.
Again, this theorem  pinpoints completely
uniqueness and existence of a $G$-module structure on
a $G$-stable $\fg$-module.

It is interesting to see whether our results could be applied to two
old conjectures in the area: 
Humphreys-Verma Conjecture \cite{Don2, PS1}, \cite[Ch. 11]{Jan}
and 
Verma Conjecture \cite{Don1,Xan}

We would like to thank Ami Braun, Simon Goodwin and Jim Humphreys 
for valuable discussions.
We are indebted to Stephen Donkin
for encouragement, interest in our work and sharing 
Xanthopoulos' thesis.

\section{$G$-stable modules for abstract groups}
\label{s2}
In this chapter we study $\bA G$-modules where $G$ is a group, $\bA$
is an associative ring. 

\subsection{Automorphisms of indecomposable modules}
\label{s2.1}
Let $\bB$ be a finite dimensional algebra over a field $\bK$
(of any characteristic), $M$ a finite-dimensional $\bB$-module, $\bE =\End(M)$ its endomorphism ring, $J=J(\bE)$ its Jacobson radical, and $H=\Aut (M)$
its automorphism group. We start with the following useful observation:
\begin{prop}
  \label{aut_mod}
  \begin{enumerate}
    \item The quotient algebra $\bE/J$ is a division algebra if and only if $M$ is indecomposable.
    \item
      If $M$ is indecomposable and $\bE /J$ is separable,
      then $H\cong \GL_1(\bD)\ltimes U$ where
      $\bD=\bE /J$ is a division algebra and
    $U=1+J$ is a connected unipotent group. 
    \item Further in the conditions of (2),
    if $\bD=\bK$, then $H= \GL_1 (\bK) \times U$.
  \end{enumerate}
\end{prop}
\begin{proof} (1) It is a standard fact that a finite length module
  is indecomposable if and only if its endomorphism ring is local.
  Since $\bE$ is finite-dimensional, this is equivalent to $\bE/J$ being
  a division ring.
  
  (2) By (1), $\bD=\bE/J$. Since $\bD$ is separable, we can use the Malcev-Wedderburn Theorem to split off the radical, i.e., to realize $\bD$ as a subalgebra
  of $\bE$ such that $\bE= \bD \oplus J$. 

  Clearly, $H=\GL_1 (\bE)$. Consider an element $x=d+j$, $d\in\bD$, $j\in J$.
Since $x^n = d^n + j^\prime$ for some $j^\prime\in J$, $x$ is nilpotent if and only if $d=0$. By the Fitting Lemma, $x\in H$ if and only if $d\neq 0$. The key isomorphism
  is given by the multiplication map:
   $$   \GL_1(\bD)\ltimes U \xrightarrow{\cong} H = \GL_1  (\bE), \ \
  (d, 1+ j) \mapsto d+dj \; , $$
  $$ H = \GL_1  (\bE) \xrightarrow{\cong} \GL_1(\bD)\ltimes U, \ \ 
  d + j \mapsto (d, 1 +d^{-1}j) \; . $$

  It remains to observe that $U=1+J$ is a connected unipotent algebraic group.
  It is connected because it is isomorphic to $J$ as a variety.
  It is unipotent because each of its elements is unipotent in $\GL (M)$.
%
%
%

  (3) The Malcev-Wedderburn decomposition turns $J$ into a
  $\bD$-$\bD$-bimodule. Our condition forces
  $\bD\otimes_\bK \bD^{op} = \bK\otimes_\bK \bK^{op} = \bK$
  so that the bimodule structure is just the $\bK$-vector space structure.
  Hence, $\GL_1(\bD)=\GL_1 (\bK)$ and $U$ commute. 
\end{proof}

\subsection{$(L,H)$-Morphs}
\label{s2.2}
Let $G\geq L$, $K\geq H$ be two group-subgroup pairs.
Let $N=N_K(H)$ and $C_K(H)$ be the normaliser and the centraliser
of $H$ in $K$. 
By {\em an $(L,H)$-morph from $G$ to $K$}
we understand
a function $f:G\rightarrow K$ satisfying the following four conditions:
\begin{enumerate}
\item $f\mid_L$ is a group homomorphism.
\item $f(G)\subset N_K(H)$.
\item $f(x)f(y)\in f(xy)H$ for all $x,y \in G$.
\item $f(L)\subset C_K(H)$.
\end{enumerate}
By a {\em weak $(L,H)$-morph from $G$ to $K$} 
we understand
a function $f:G\rightarrow K$ satisfying only the first three conditions. 

One can observe that a weak $(L,H)$-morph is just
a homomorphism $G\rightarrow N/H$ with a choice of lifting to $N$
satisfying an additional condition.
For instance, weak $(G,1)$-morphs are the same as homomorphisms $G\rightarrow K$
and weak $(1,K)$-morphs are just functions $G\rightarrow K$ which preserves the identity. Furthermore, the same statements also hold if we replace weak morphs with morphs in the previous sentence.

Commonly $(L,H)$-morphs originate from 
$K$-$G$-sets $X=\,_KX_G$, i.e., 
$G$ acts on the right, $K$ on the left and the actions commute.
Let $\theta\in X$ such that its $G$-orbit is inside its $K$-orbit.
Let $H$ be the stabiliser of $\theta$ in $K$. Choose a section
$K/H\rightarrow K$ which sends the coset $H$ to $1_K$. The composition of
the section with the $G$-orbit map of $\theta$ is a function
$$
f:G\rightarrow K \ \ \mbox{ characterised by } \ \ 
\,^{f(x)}\theta = \theta^{x}
\ \ \mbox{ for all } x\in G.$$
\begin{lemma}
\label{1Hhom}
The map  $f$ defined above is a $(1,H)$-morph.
\end{lemma}
\begin{proof}
	 
By definition,  $\,^{f(xy)}\theta = \theta^{xy}$.
  On the other hand, 
  $\theta^{xy}= (\theta^x)^y = (\,^{f(x)}\theta)^y =
  \,^{f(x)f(y)}\theta$.
  Hence, $\theta = \,^{f(xy)^{-1}f(xy)}\theta
  = \,^{f(xy)^{-1}f(x)f(y)}\theta$
  and $f(xy)^{-1}f(x)f(y)\in H$.

  Now pick $h\in H$. Then
$\,^{f(x)^{-1}h f(x)}\theta =
  \,^{f(x)^{-1}h}\theta^x =
  \,^{f(x)^{-1}}\theta^x =
  \,^{f(x)^{-1}f(x)}\theta= \theta$
  so that 
  ${f(x)^{-1}h f(x)}\in H$.
\end{proof}
We would like to identify weak $(L,H)$-morphs that
define the same homomorphisms $G\rightarrow N/H$.
More precisely, we say that two weak $(L,H)$-morphs $f$ and $f^\prime$ are equivalent
if $f^\prime (x)\in f(x)H$ for all $x \in G$. 
We denote the set of equivalence classes of weak $(L,H)$-morphs by
$[LH]{\mbox{\rm mo}}(G,K)$.
Furthermore, given a fixed homomorphism $\theta:L\rightarrow K$ we denote by
$\LH(G,K)$ the set of equivalence classes of those
weak $(L,H)$-morphs that restrict to $\theta$ on $L$.

Let $A$ be an additive abelian group with a $G$-action (a $\bZ G$-module).
We consider a subcomplex $(\widetilde{C}^\bullet (G,L;A),d)$
of the standard complex $(C^\bullet (G;A),d)$
that consists of such cochains $\mu_n$
that are trivial on $L^n$, i.e., $\mu_n\mid_{L\times \ldots \times L} \equiv 0_A$.

We observe that this cochain complex fits into an exact sequence of cochain complexes
$$0\rightarrow \widetilde{C}^\bullet(G,L;A)\rightarrow C^\bullet (G;A)\rightarrow C^\bullet (L;A)\rightarrow 0 \; .$$
This then allows us to form a long exact sequence of cohomology
$$\ldots\rightarrow H^{n-1}(G;A)\rightarrow H^{n-1}(L;A)\rightarrow \widetilde{H}^n(G,L;A)\rightarrow
H^n(G;A)\rightarrow H^n(L;A)\rightarrow\ldots$$

For our purposes, we have to modify this subcomplex slightly. We consider a subcomplex $(C^\bullet (G,L;A),d)$
of the standard complex $(C^\bullet (G;A),d)$ which is obtained from $(\widetilde{C}^\bullet (G,L;A),d)$ in the following way: 
for $n>0$, $C^n(G,L;A)=\widetilde{C}^n (G,L;A)$, 
whilst $C^0(G,L;A)=A^L$. We can furthermore replace the complex $C^\bullet(L;A)$ with the 
complex $\widetilde{C}^\bullet(L;A)$, defined by $\widetilde{C}^n(L;A)=\Coker(C^n(G,L;A)\to C^n(G;A))$ for all $n\geq 0$. In particular, we observe that 
$\widetilde{C}^n(L;A)=C^n(L;A)$ for all $n\geq 1$. This then recovers 
an exact sequence of cochain complexes: 
$$0\rightarrow C^\bullet(G,L;A)\rightarrow C^\bullet (G;A)\rightarrow \widetilde{C}^\bullet (L;A)\rightarrow 0 \; .$$

In particular, observing that for the cochain complex $\widetilde{C}^\bullet(L;A)$ we have $\widetilde{H}^0(L;A)=0$ and $\widetilde{H}^n(L;A)=H^n(L;A)$ for $n\geq 1$, 
we can form the long exact sequence of cohomology
$$0\rightarrow H^1(G,L;A)\rightarrow\ldots\rightarrow H^{n-1}(L;A)\rightarrow H^n(G,L;A)\rightarrow
H^n(G;A)\rightarrow H^n(L;A)\rightarrow\ldots$$

What can we say about the natural map $f_n :H^n (G,L;A) \rightarrow H^n(G;A)$?
From this long exact sequence, the following proposition is clear.

\begin{prop}
	\label{cohom_compare}
	\begin{enumerate}
		\item For $n>0$, $H^{n} (L;A)=0$ if and only if $f_n$ is surjective and $f_{n+1}$ is injective.
		\item For $n>1$, $f_n$ is injective if and only if the restriction map $Z^{n-1} (G;A)\rightarrow Z^{n-1} (L;A)$ is surjective.
	\end{enumerate}
\end{prop}
\begin{proof}
	%
	
	{(1)} This follows from the exact sequence.
	
	{(2)} Suppose $Z^{n-1} (G;A)\rightarrow Z^{n-1} (L;A)$ is surjective.
	Pick $\mu \in Z^n(G,L;A)$ such that $[\mu]\in\ker (f_n)$.
	Then $\mu\in B^n(G;A)$ and $\mu = d \eta$ for some $\eta\in C^{n-1}(G;A)$.
	Moreover, $d (\eta|_{L})=\mu|_{L}\equiv 0$ so that $\eta|_{L}\in Z^{n-1}(L;A)$.
	Our assumption gives $\zeta\in Z^{n-1}(G;A)$ such that $\zeta|_L=\eta|_L$.
	Hence, $\eta-\zeta \in C^{n-1}(G,L;A)$ and $\mu = d(\eta-\zeta) \in B^{n}(G,L;A)$.
	
	Now suppose $f_n$ is injective. Pick $\mu\in Z^{n-1}(L;A)$, and extend it to $\chi\in C^{n-1}(G;A)$. Hence $d\chi\in Z^n(G,L;A)$ and $[d\chi]\in \ker(f_n)$. So $d\chi=d\zeta$ for some $\zeta\in C^{n-1}(G,L;A)$. Now $\chi-\zeta\in Z^{n-1}(G;A)$ and $(\chi-\zeta)|_L=\mu$.
\end{proof}



\begin{cor}
	For $n>1$, $H^n(G,L;A)=0$ if and only if $H^{n-1}(G;A)\rightarrow H^{n-1}(L;A)$ is surjective and $H^{n}(G;A)\rightarrow H^{n}(L;A)$ is injective. Furthermore, $H^1(G,L;A)=0$ if and only if $H^{1}(G;A)\rightarrow H^{1}(L;A)$ is injective.
\end{cor}

The next theorem clarifies the origin of this new complex.
Let us fix a homomorphism 
$\theta=f|_L : L \rightarrow N$ 
and choose a subgroup 
$\widetilde{H} \leq H$, normal in $N=N_K(H)$ such that $A\coloneqq H/\widetilde{H}$ is abelian.
Notice that the conjugation 
$\,^{g H}h \widetilde{H} \coloneqq ghg^{-1} \widetilde{H}$
defines a structure of an $N/H$-module (and a $G$-module via any weak $(L,H)$-morph) on $A$. 
Informally, we should think of the next theorem as ``an exact sequence''
\begin{equation}
  \label{exact_seq0}
H^1(G,L;A)
\dashrightarrow 
[L\widetilde{H}]^\theta{\mbox{\rm mo}} (G,N)
\longrightarrow 
\LH(G,N)
\longrightarrow 
H^2(G,L;A)
\end{equation} 
keeping in mind that the second and the third terms are sets (not even pointed sets)
and the first arrow is an ``action'' rather than a map.
Let us make it more precise: a weak $(L,H)$-morph defines a $G$-module structure $\rho$ on $A$.
For each particular $\rho$ (not just its isomorphism class) we define
$$
[L\widetilde{H}]^\theta{\mbox{\rm mo}} (G,N)_\rho \subseteq
[L\widetilde{H}]^\theta{\mbox{\rm mo}} (G,N), \ \ \ 
\LH(G,N)_\rho
\subseteq
\LH(G,N)
$$
as subsets of those weak $(L,H)$-morphs that define this particular $G$-action $\rho$.
These subsets could be empty, in which case we consider the following theorem true for trivial reasons. The reader should consider this theorem and its proof as a generalisation of the results in sections 1 and 2 in \cite{Thev} to the situation of weak $(L,H)$-morphs.
\begin{theorem}
  \label{exact_seq}
We are in the notations preceding this theorem. For each $G$-action $\rho$ on $A$ 
the following statements hold:
\begin{enumerate}
\item There is a restriction map 
$$
\Res:
[L\widetilde{H}]^\theta{\mbox{\rm mo}} (G,N)_\rho
\longrightarrow 
\LH(G,N)_\rho, \ \ \
\Res (\langle f\rangle ) = [f]
$$
where 
$\langle f\rangle$ and $[f]$
are the equivalence classes 
in 
$[L\widetilde{H}]^\theta{\mbox{\rm mo}} (G,N)_\rho$
and
$\LH(G,N)_\rho$.
\item The abelian group $Z^1(G,L;(A,\rho))$
acts freely on the set $[L\widetilde{H}]^\theta{\mbox{\rm mo}} (G,N)_\rho$
by
$$
\gamma \cdot \langle f \rangle \coloneqq \langle \dot{\gamma} f \rangle
\ \mbox{ where } \ 
\dot{\gamma}f (x) = \dot{\gamma} (x) f (x)
\ \mbox{ for all } \ x\in G 
$$  
and $\dot{\gamma}: G \xrightarrow{\gamma} A \rightarrow H$
is a lift of $\gamma$ to a map $G\rightarrow H$ with
$\dot{\gamma} (1)=1$.
\item The corestricted restriction map 
$
\Res:
[L\widetilde{H}]^\theta{\mbox{\rm mo}} (G,N)_\rho
\longrightarrow 
\mbox{\rm Im}(\Res) 
$
is a quotient map by the $Z^1(G,L;(A,\rho))$-action.  
\item Two classes $\langle f\rangle, \langle g\rangle\in[L\widetilde{H}]^\theta{\mbox{\rm mo}} (G,N)_\rho$ 
lie in the same $B^1(G,L;(A,\rho))$-orbit if and only if
there exist $h\in H$, $f^\prime\in \langle f\rangle$, $g^\prime\in \langle g\rangle$
such that $[f(L),h]\subset\widetilde{H}$ and $f^\prime (x) = h g^\prime (x) h^{-1}$
for all $x \in G$.
\item
There is an obstruction map
$$\Obs:
\LH(G,N)_\rho
\longrightarrow 
H^2(G,L;(A,\rho)), \ \ \
\Obs ([f]) = [f^\sharp]
$$
where the cocycle $f^\sharp$ is defined by
$
f^\sharp(x,y)= f(x)f(y)f(xy)^{-1}\widetilde{H}
$.
\item The sequence (\ref{exact_seq0}) is exact, i.e., the image of $\Res$ is equal to $\Obs^{-1}([0])$.
\end{enumerate}
\end{theorem}
\begin{proof}
Suppose $\langle f\rangle = \langle g\rangle$. This gives a function
$\alpha :G \rightarrow \widetilde{H}$ such that $\alpha |_L \equiv 1$ and
$f(x) = \alpha (x) g(x)$ for all $x \in G$. Since $H\supseteq \widetilde{H}$,
we conclude that $[f]=[g]$ and the map $\Res$ is well-defined. This proves (1).

Suppose $\Res (\langle f\rangle) = \Res (\langle g\rangle)$.
Then 
$[f] = [g]$ gives a function
$\alpha :G \rightarrow H$ such that $\alpha |_L \equiv 1$ and
$f(x) = \alpha (x) g(x)$ for all $x \in G$. 
We can also obtain such a function from a cochain
$\gamma \in C^1(G,L;(A,\rho))$ by lifting
$\alpha = \dot{\gamma}$.  
Let us compute in the group $N/\widetilde{H}$
denoting
$a\widetilde{H}$ by $\overline{a}$.
The weak $(L,H)$-morph condition for $f$ is
equivalent to
the following equality:
$$
\overline{\alpha (xy)} \; \overline{g(xy)} 
= 
\overline{f (xy)}
=
\overline{f (x)} \; \overline{f(y)}
=
\overline{\alpha (x) g(x)} \; 
\overline{\alpha (y) g(y)}
=
\overline{
\alpha (x) g(x)
\alpha (y)  g(x)^{-1}} \; \overline{g(x) g(y)}.
$$
Now notice that 
$$
\overline{g(xy)} 
= 
\overline{g(x) g(y)}
=
\overline{g(x)} \; \overline{g(y)}
$$
is the weak $(L,H)$-morph condition for $g$, while
$$
\overline{\alpha (xy)} 
=
\overline{\alpha (x) g(x)
\alpha (y)  g(x)^{-1}}
=
\overline{\alpha (x)}\;
\overline{g(x)\alpha (y)  g(x)^{-1}}
=
\overline{\alpha (x)}\;
[\rho (x) (\overline{\alpha})] (y)  
$$
is the cocycle condition for $\overline{\alpha}=\alpha \widetilde{H}$.
Any two of these three conditions imply the third one,
which proves both (2) and (3), except the action freeness.

Suppose $\langle f\rangle=\gamma\cdot\langle f\rangle=\langle \dot{\gamma}f\rangle$.
This gives a function
$\alpha :G \rightarrow \widetilde{H}$ such that $\alpha |_L \equiv 1$ and
$\dot{\gamma}(x)f(x) = \alpha (x) f(x)$ for all $x \in G$.
Hence, $\dot{\gamma}=\alpha$ and
$\gamma=\overline{\alpha}\equiv 1$.
Thus, the action is free.

Let us examine
$da\cdot \langle f\rangle=
\langle \dot{da}f\rangle$ for some $a\in A^L$.
Since $d a (x) = -a +\rho (x)(a)$
and $\rho (x)$ can be computed by conjugating with $f(x)$, 
we immediately conclude that
$$
[\dot{da}f] (x) =
\dot{a}^{-1} f(x)\dot{a}f(x)^{-1}f(x) =
\dot{a}^{-1} f(x)\dot{a}.
$$
It is easy to see that $[f(L),\dot{a}]\subset \widetilde{H}$.
The argument we have just given is reversible, i.e., if
$f (x) = h g (x) h^{-1}$ then $\langle g\rangle = d\overline{h} \cdot \langle f \rangle$ and $\overline{h}\in A^L$. 
This proves (4).

Suppose 
$[f] = [g]$. This gives a function
$\alpha :G \rightarrow H$ such that $\alpha |_L \equiv 1$ and
$f(x) = \alpha (x) g(x)$ for all $x \in G$. 
Let us compute the cocycles in $N/\widetilde{H}$, 
keeping in mind that $H/\widetilde{H}$ is abelian:
\begin{align*}
f^\sharp (x, y) =  \overline{f(x)}\overline{f(y)}\overline{f(xy)^{-1}}
= & \ 
\overline{\alpha(x)}\; \overline{g(x)}\; \overline{\alpha(y)} \; \overline{g(y)} \;
\overline{g(xy)}^{-1} \overline{\alpha(xy)}^{-1} 
= \\
(\overline{\alpha(xy)}^{-1} \overline{\alpha(x)}\; \overline{g(x)} 
\overline{\alpha(y)} \overline{g(x)^{-1}})
\overline{g(x)}\overline{g(y)} \overline{g(xy)^{-1}}
& = 
d\,\overline{\alpha} (x,y) + g^\sharp(x, y).
\end{align*}
Thus $[f^\sharp] = [g^\sharp]$, proving (5).
 
It is clear that $f^\sharp \equiv 1$ for 
$f\in [L\widetilde{H}]^\theta{\mbox{\rm mo}} (G,N)_\rho$.
Hence, $\Obs (\Res (\langle f \rangle))=[0]$.
Suppose now that $\Obs ( [f])=[0]$.
This gives a function
$\alpha :G \rightarrow H$ such that $\alpha |_L \equiv 1$ and
$d \overline{\alpha} = f^\sharp$
Consider $g:G\rightarrow N$ defined by $g(x) =  \alpha (x)^{-1}f(x)$ for all $x \in G$.
Then $[g]=[f]$ and we can verify that $g\in [L\widetilde{H}]^\theta{\mbox{\rm mo}} (G,N)_\rho$
by checking $g^\sharp \equiv 1$ in $N/\widetilde{H}$:
\begin{align*}
g^\sharp (x,y)=
\overline{\alpha(x)}^{-1}
\overline{f(x)}\;
\overline{\alpha(y)}^{-1} 
\overline{f(y)}\;
\overline{f(xy)}^{-1}
\overline{\alpha(xy)}\; 
&
\sim 
\overline{\alpha(xy)}\; 
\overline{\alpha(x)}^{-1}
(\overline{f(x)}\;
\overline{\alpha(y)}
\; \overline{f(x)}^{-1} )^{-1} 
f^\sharp (x,y)
\\
=
(d\,\overline{\alpha}(x,y))^{-1}
f^\sharp (x,y)
&
\equiv 1.
\end{align*}
This proves (6). 
\end{proof}

Let us quickly re-examine how the last section works
for $(L,H)$-morphs. 
All of its results including
Theorem~\ref{exact_seq} clearly work,
although the objects that appear have additional properties.
Most crucially,
since $f(L)\subseteq C_K(H)$,
the $L$-action
on the abelian group $A$ is trivial.
If $L$ is normal in $G$,
this just means that $A$ is a $\bZ G/L$-module.
%

An important feature is that $Z^1(L;A)$ consists of homomorphisms $L\rightarrow A$ in this case.
This means that Proposition~\ref{cohom_compare}
yields the following corollary:  
\begin{cor}
  If the group $L$ is perfect, then
  $f_1 :H^1 (G,L;A) \rightarrow H^1(G;A)$ is surjective
  and 
  $f_2 :H^2 (G,L;A) \rightarrow H^2(G;A)$ is injective.
 \end{cor}

\subsection{Module extensions}
\label{s2.4}

We now assume that $L$ is a normal subgroup of $G$.
Let $\bA$ be an associative ring, $(V,\theta)$ an
$\bA L$-module, $K=\Aut_{\bA}V$ and $H=\Aut_{\bA L}V$
its automorphism groups. We can think of $\theta$
as an element of the set of $\bA L$-structures
$X= \hom (L,K)$. Then $H$ is the centraliser in $K$ of $\theta (L)$.
By $N$, as before, we denote the normaliser of $H$ in $K$.

Naturally, $X$ is a $K$-$G$-set: $G$ acts by conjugation on $L$
twisting the $\bA L$-module structure. $K$ acts by conjugations
on the target, while $H=\Sta_K (\theta)$.
The module $V$ is called {\em $G$-stable}
if $(V,\theta)\cong (V,\theta^g)$ for all $g\in G$.
This is equivalent to the orbit inclusion $\theta^G\subseteq \,^K\theta$.
By Lemma~\ref{1Hhom} this gives a $(1,H)$-morph
$f : G \rightarrow K$.

If $g \in L$, the isomorphism $f(g) : (V,\theta) \rightarrow (V,\theta^g)$
can be chosen to be $\theta (g)$. Indeed,
$$
\theta (g) (\theta(h) v)=\theta (gh) (v) =
\theta (ghg^{-1}) (\theta(g)(v)) =
\theta^{g}(h)(\theta(g)(v))
$$ 
for all $g,h\in L$.
Then, without loss of generality $f\vert_L=\theta$,
and $f$ is an $(L,H)$-morph in $\LH(G,N)$.

Suppose that the group $H=\Aut_{\bA L}V$ is soluble.
We can always find its subnormal series
$H=H_0 \rhd H_1 \rhd \ldots \rhd H_k = \{1\}$
with abelian quotients $A_j=H_{j-1}/H_j$ such that
each $H_j$ is normal in $N$. For instance, we can
use the commutator series $H_j=H^{(j)}$.
In this case, every abelian group $A_j$ becomes an $N$-module.

If $\bA$ is finite-dimensional over the field $\bK$
and $V$ is a finite-dimensional indecomposable
$\bA L$-module, we can use Proposition~\ref{aut_mod}
to derive useful information about its automorphisms.
In particular, if $\bD=\End_{\bA L}(V)/J$ is a separable field extension of $\bK$,
then $H= \GL_1 (\bD) \ltimes (1+J)$ is soluble.
It admits another standard
$N$-stable subnormal series:
$$
H_m = 1+J^m, \ m\geq 1, \ \ \
A_m = (1+J^m)/(1+J^{m+1}). 
$$
	As groups, we have $A_m=((1+J^m)/(1+J^{m+1}),\cdot)\cong (J^m/J^{m+1},+)$.
The following theorem is the
direct application of Theorem~\ref{exact_seq}.
It determines the uniqueness and existence of a $G$-module structure on a $G$-stable $L$-module.
The proof is obvious.
\begin{theorem}\label{branching}
  Let $V=(V,\theta)$ be a $G$-stable $\bA L$-module with a soluble automorphism
  group $H$, where $\bA$ is an associative ring. Let $H=H_0 \rhd H_1 \rhd \ldots \rhd H_k = \{1\}$
  be a subnormal $N$-stable series with abelian factors $A_j=H_{j-1}/H_j$.

  Any $\bA G$-module structure $\Theta$ on $(V,\theta)$ compatible
  with its $\bA L$-structure (i.e., $\Theta |_{\bA L} = \theta$)
  can be discovered by the following recursive process in $k$ steps.
  One initialises the process with an $(L,H_0)$-morph $f_0=f$
  coming from the $G$-stability. The step $m$ is the following.
  \begin{enumerate}
  \item The $(L,H_{m-1})$-morph $f_{m-1}:G\rightarrow N$ such that
    $f_{m-1}|_L=\theta$ determines a $G$-module structure $\rho_m$ on $A_m$.
  \item If $\Obs ([f_{m-1}])\neq 0\in H^2 (G, L;(A_m,\rho_m))$, then
    this branch of the process terminates.
  \item If $\Obs ([f_{m-1}])=0\in H^2 (G, L;(A_m,\rho_m))$, then
    we choose an $(L,H_{m})$-morph $f_{m}:G\rightarrow N$
    such that $\Res ([f_m])=[f_{m-1}]$.
  \item For each element of $H^1 (G, L;(A_m,\rho_m))$ we choose
    a different $f_m$ branching the process. (The choices
    different by an element of $B^1 (G, L;(A_m,\rho_m))$
    are equivalent, not requiring the branching.)
  \item We change $m$ to $m+1$ and go to step (1).
  \end{enumerate}
  An $\bA G$-module structure $\Theta$ on $(V,\theta)$ compatible
  with its $\bA L$-structure
  is equivalent to $f_k$ for one of the non-terminated branches.
  Distinct non-terminated branches produce (as $f_k$)
  non-equivalent compatible $\bA G$-module structures.
\end{theorem}
This process is subtle as $\rho_m$ is revealed only
when $f_{m-1}$ is computed.
It would be useful to
have stability, i.e., the fact the $G$-modules
$(A_m,\rho_m)$ are the same (isomorphic) for different
branches. The actions $\rho_m$ on $A_m=H_{m-1}/H_m$
on different branches
differ by conjugation via a function $G\rightarrow H_{m-2}$.
Thus, one needs all two-step quotients $H_{m-1}/H_{m+1}$
to be abelian to ensure stability.
Having said that, we can still have some easy criteria
for existence, uniqueness and non-uniqueness.

\begin{cor} {\rm (Existence Test)}
  Suppose $H^2 (G, L;(A_m,\rho_m))=0$ for all $m$ for one of the branches.
  Then this branch does not terminate and an $\bA G$-module structure exists.
\end{cor}

\begin{cor}\label{Uniq} {\rm (Uniqueness Test)}
  Suppose $H^1 (G, L;(A_m,\rho_m))=0$ for all $m$
  for one of the non-terminating branches.
  Then this branch is the only branch. Moreover, if 
  an $\bA G$-module structure exists, it is unique up to an isomorphism.
\end{cor}

\begin{cor} {\rm (Non-Uniqueness Test)}
  Suppose $H^1 (G, L;(A_k,\rho_k))\neq0$ 
  for one of the non-terminating branches.
  Then there exist non-equivalent $\bA G$-module structures.
  \end{cor}

%
%
%

\subsection{Extension from not necessarily normal subgroups}

In Section \ref{s2.4} we restrict our attention to the case of $L$ being a normal subgroup of $G$. Let us take a moment to examine how Section \ref{s2.4} works if $L$ is not normal.

Set $P\coloneqq\bigcap_{g\in G}L^g$, where $L^g\coloneqq g^{-1}Lg$. 
Let $\bA$ be an associative ring, $(V,\theta)$ an
$\bA L$-module. Note that $(V,\theta)$ is also an $\bA P$-module under restriction, so we can view $\theta$
as an element of the set
$X= \hom (P,K)$. Let $K=\Aut_{\bA}V$ and $H=\Aut_{\bA P}V$ be
its automorphism groups, so $H$ is the centraliser in $K$ of $\theta (P)$.
By $N$, as before, we denote the normaliser of $H$ in $K$.

As in Section \ref{s2.4}, $X$ is a $K-G$-set.
The $\bA L$-module $V$ is called {\em $G$-stable-by-conjugation}
if $(V,\theta)\cong (V,\theta^g)$ as $\bA[L\cap L^g]$-modules for all $g\in G$. Note that this condition guarantees that $V$ is $G$-stable as an $\bA P$-module.
This is equivalent to the orbit inclusion $\theta^G\subseteq \,^K\theta$.
By Lemma~\ref{1Hhom} this gives a $(1,H)$-morph
$f : G \rightarrow K$.

If $g \in L$, the $\bA[L\cap L^g]$-isomorphism $f(g) : (V,\theta) \rightarrow (V,\theta^g)$
can be chosen to be $\theta (g)$. Indeed,
$\theta (g) (\theta(h) v)=\theta (gh) (v) =
\theta (ghg^{-1}) (\theta(g)(v)) =
\theta^{g}(h)(\theta(g)(v))$ for $g\in L$, $h\in L\cap L^g$.
Then, without loss of generality $f\vert_L=\theta$,
and $f$ is an $(L,H)$-morph in $\LH(G,N)$.

This then allows us to proceed with the inductive process of Theorem \ref{branching} as before, when $H=\Aut_{\bA P}V$ is soluble.

\subsection{Comparison with $C^\bullet(G/L;A)$}
\label{s2.5}

When studying the question of extending representations from a normal subgroup, Dade and Th\'{e}venaz use the cohomology of the cochain complex $(C^\bullet(G/L;A),d)$ to control existence and uniqueness of such extensions.
In this paper, however, we use the cohomology complex
$(C^\bullet(G,L;A),d)$ instead. It is worth taking a moment
to compare the cohomology of these two complexes, and see where the difference in approaches arises. We use the notation of Section~\ref{s2.2},
assuming that cochains are normalised
since this does not affect the cohomology groups.

In order for the action of $G/L$ on $A$ to make sense,
we need to make the assumption that $L$ acts on $A$ trivially.
The reader can observe that this assumption holds in the case considered in Section~\ref{s2.4}, and, in fact, holds whenever one obtains the $G$-action on $A$
from an $(L,H)$-morph as opposed to a weak $(L,H)$-morph. With this assumption, we have the following proposition.

\begin{prop}\label{H1Map}
	Under the aforementioned conditions
	we have isomorphisms of groups
	\newline 
	$H^0(G,L;A) \cong H^0 (G/L;A)$
	and
	$H^1 (G,L;A)\cong H^1 (G/L;A)$.
\end{prop}
\begin{proof}
	It is easy to see that 
	$H^0(G,L;A) = A^G =  H^0 (G/L;A)$.
	The natural map from the group of normalised cochains
	$$
	\inf : \widehat{C}^1(G/L;A) \rightarrow   C^1 (G,L;A), \ \ \ 
	\inf (\mu) (g) = \mu (g L).
	$$
	defines a map 
	$\Inf\coloneqq[\inf] : H^1(G/L;A) \rightarrow   H^1 (G,L;A)$
	of cohomology groups.
	It is injective because $\Inf([\mu])=0$ means that $\inf (\mu) = da$ for some $a\in A$.
	Then $\mu = da$ and $[\mu]=0$.
	
	It is surjective because for $\eta \in Z^1 (G,L;A)$ we have $d\eta =0$ that translates as
	$$
	\eta (gh) = \,^g(\eta (h)) + \eta (g) \ \ 
	\mbox{ for all } \ \ 
	g,h\in G.
	$$
	If one chooses $h\in L$, then it tells us that $\eta (gh) =  \eta (g)$,
        i.e., that $\eta$ is constant on $L$-cosets.
	Thus, the cocycle  
	$$
	\mu \in  \widehat{Z}^1(G/L;A) , \ \ \ 
	\mu (g L) \coloneqq \eta (g)
	$$
	is well-defined. By definition $\inf (\mu) = \eta$. 
\end{proof}

Considering the second cohomology of these complexes,
it is still possible to construct the inflation map $\Inf:H^2(G/L;A)\to H^2(G,L;A)$ in the natural way, but this map is no longer an isomorphism in general. We can still view  $H^2(G/L;A)$ as a subgroup of $H^2(G,L;A)$:

\begin{prop}\label{H2inj}
	The map $\Inf:H^2(G/L;A)\to H^2(G,L;A)$ is injective.
\end{prop}

	\begin{proof}
	  If $\Inf([\eta])=0\in H^2(G,L;A)$ then there exists $\mu\in C^1(G,L;A)$ such that $d\mu=\inf(\eta)$. Note that $\inf(\eta)$ is constant on $L\times L$-cosets by construction. In particular, for $g\in G$ and $h\in L$, we have
          $$\mu (g) - \mu (gh) = \,^g(\mu (h)) + \mu (g) - \mu (gh) =  \inf(\eta)(g,h)=\inf(\eta)(g,1)=\inf(\eta)(1,1)=0\, ,$$
          using the cocycle condition in the penultimate equality.
          Hence, $\mu$ is constant on cosets of $L$ in $G$. In particular,
          if we define $\widetilde{\mu}\in \widehat{C}^1(G/L;A)$ by $\widetilde{\mu}(gL)=\mu(g)$ then we obtain that $\eta=d\widetilde{\mu}$ and so $[\eta]=0\in H^2(G/L;A)$.
	\end{proof}

        In the context of Theorem~\ref{exact_seq}, we can see that $H^2(G/L;A)$ and $H^2(G,L;A)$ can be made to play the same role in certain key cases.
        To that end, we say that an $(L,H)$-morph $f$ is {\em normalised} if $f(gh)=f(g)f(h)$ whenever $g\in G$ and $h\in L$. Note that this definition is independent of the subgroup $H$.

\begin{lemma}\label{Norm}
		In the context of Theorem~\ref{branching}, the $(L,H_i)$-morphs $f_i$ can be assumed to be normalised for each $i$. Furthermore, with this assumption, the cocycles $f_i^\sharp \in Z^2(G,L;A_{i+1})$ are constant on cosets of $L\times L$ in $G\times G$.
\end{lemma}

\begin{proof}
	These results follow easily from Lemma 9.2 and Lemma 9.4(i) in Karpilovsky \cite{Kar}.
\end{proof}

For the remainder of this section 
we assume that all morphs are normalised. The second statement of
Lemma~\ref{Norm}
immediately yields that, given an $(L,H)$-morph $f$, $\Obs([f])$ lies in the image of the natural homomorphism $\Inf:H^2(G/L;A)\to H^2(G,L;A)$.
The discussion in this section yields the following result.

\begin{cor}\label{H2Map}
  Let $f$ be a normalised $(L,H)$-morph. Then there exists $\eta\in Z^2(G/L;A)$
  with $\Inf([\eta])=\Obs([f])$. Furthermore, $\Obs([f])=0\in H^2(G,L;A)$ if and only if $[\eta]=0\in H^2(G/L;A)$.
\end{cor}

Combining Proposition~\ref{H1Map} and Corollary~\ref{H2Map}, we observe that Sections~\ref{s2.2} and \ref{s2.4} could be interpreted using the cochain complex $C^\bullet(G/L;A)$ at all points instead of the complex $C^\bullet(G,L;A)$ (although doing so would force us to work exclusively with normalised morphs instead of not-necessarily-normalised weak morphs). Indeed, this is the approach taken by Dade and Th\'{e}venaz in the contexts they consider. Our reasons for not taking this approach are threefold. Firstly, our new complex fits nicely into an exact sequence as described in Section~\ref{s2.2}.
Secondly, this complex is more natural to work with --
Dade and Th\'{e}venaz essentially move from the complex $C^\bullet(G/L;A)$ to the complex $C^\bullet(G,L;A)$ as described in this section, and then proceed as we do.
Finally, our main motivation in studying the case for abstract groups is to gain insight into the question for algebraic groups, where the procedures described in this section do not work smoothly (cf. Section~\ref{s3.5}). 

In particular, the reader should note that if $H$ is abelian then the corollaries at the end of Section \ref{s2.4} give precisely Corollary 1.8 and Proposition 2.1 in \cite{Thev}.

\section{$G$-stable modules for algebraic groups}
\label{s3}

In this chapter we consider algebraic groups over an algebraically closed field $\bK$ of positive characteristic $p$. Algebraic groups are affine and reduced,
groups schemes are affine and not necessarily reduced.

\subsection{Rational and algebraic $G$-modules}
\label{s3.1}
We distinguish algebraic and rational maps of algebraic varieties.
In particular, we can talk about algebraic and rational homomorphisms
of algebraic groups $f:G\rightarrow H$.
The latter are defined on an open dense subset $U=\dom(f)$ of $G$ containing $1$ and satisfy
$f(x)f(y)=f(xy)$ whenever $x,y,xy\in U$.

A rational automorphic $G$-action on a commutative algebraic group $H$ is a rational map $G\times H\rightarrow H$, defined on an open set $U\times H$ containing $1\times H$, with the usual action conditions and also such that for each $g\in U$ the map $x\mapsto \,^gx$ is a group automorphism of $H$. An algebraic $G$-action on $H$ is the same, but where the map $G\times H\rightarrow H$ is algebraic.

In an important case, the distinction between rational and algebraic maps can be essentially forgotten, as observed by Rosenlicht \cite{Ros2}.

\begin{lemma}\cite[Theorem 3]{Ros2}
	\label{rat_hom}
	Let $G$ and $H$ be algebraic groups with $G$ connected. 
	Suppose  $f:G\rightarrow H$
	is a rational homomorphism.
	Then $f$ extends uniquely to an algebraic group homomorphism $G\rightarrow H$.
\end{lemma}

When $H$ is commutative, this lemma is a special case of the next lemma. Indeed, if one takes the $G$-action on $H$ to be trivial, then the condition in the following lemma is precisely the condition for a map to be a homomorphism.

\begin{lemma}\label{rat_cocycle}
	Suppose that $G$ is a connected algebraic group and $(H,+)$ is a commutative algebraic group with an algebraic automorphic $G$-action $\rho$. Let $f:G\rightarrow H$ be a rational map such that $f(xy)=f(x)+ \,^x{f(y})$ for all $x,y,xy\in \dom(f)$ (where $\,^x f(y)\coloneqq\rho(x)(f(y))$). Then $f$ extends to an algebraic map satisfying $f(xy)=f(x)+ \,^x{f(y})$ for all $x,y\in G$.
\end{lemma}

\begin{proof}
	Since $f$ is rational and $G$ is connected, $\dom(f)=U\subset G$ is a dense open subset. Set $V=U\cap U^{-1}$.
	
	Fix $x\in V$. Consider the rational map	
	$$f_x:G\rightarrow H, \quad \quad \quad f_x(y)\coloneqq f(yx) + \,^{yx}f(x^{-1}).$$
	This map is rational since it is defined on the dense open set $Vx^{-1}$.
	Observe that on $V\cap Vx^{-1}$ we have that $f_x=f$ by the assumption on $f$.
	Now, let $x,z\in V$ and define the rational map
	$$f_{x,z}:G\rightarrow H, \quad \quad \quad f_{x,z}(y)\coloneqq f_x(y)-f_z(y).$$
	
	Then $f_{x,z}$ is defined on $Vx^{-1} \cap Vz^{-1}$. If the set $f_{x,z}^{-1}(H\setminus\{0\})$ is non-empty, it is open dense. Hence, it has non-empty intersection with $V\cap Vx^{-1}\cap Vz^{-1}$. However, since on $V\cap Vx^{-1}\cap Vz^{-1}$ we have $f=f_x=f_z$, this is impossible. Thus, we must have $f_{x,z}\equiv 0$ on $Vx^{-1}\cap Vz^{-1}$. In particular, if $y\in Vx^{-1}\cap Vz^{-1}$ then $f_x(y)=f_z(y)$.
	
	Therefore, the following map is a well-defined locally-algebraic, and hence algebraic, map
	$$\widehat{f}:G\rightarrow H, \quad \quad \quad \widehat{f}(y)\coloneqq f_w(y) \ \ \mbox{where} \ \ w\in y^{-1}V.$$
	This map clearly restricts to $f$ on $V$. Furthermore, it satisfies the condition from the lemma:
	
	Let $a,b\in G$. Choose $w\in b^{-1}a^{-1}V\cap b^{-1}V$ -- this exists since both these sets are open dense in $G$. We then have $abw\in V$ and $bw\in V$. The condition on $f$ tells us that
	$0=f(1)=f(bw)+ \,^{bw}f(w^{-1}b^{-1})$.
	Hence, we have the equations
	$$\widehat{f}(ab)=f_w(ab)=f(abw)+\,^{abw}f(w^{-1}),$$
	$$\widehat{f}(a)=f_{bw}(a)=f(abw)+\,^{abw}f(w^{-1}b^{-1}),$$
	$$\,^a\widehat{f}(b)=\,^a f_w(b)=\,^{a}f(bw)+ \,^{abw}f(w^{-1}).$$	
	This then gives us that $\widehat{f}(ab)=\widehat{f}(a)+ \,^{a}\widehat{f}(b)$, as required.
	
\end{proof}

Recall that a rational\footnote{It is a standard terminology,
  which slightly disagrees with our usage of the adjective {\em rational}.}
representation of an algebraic group $G$ is a vector space $V$, equipped
with an algebraic homomorphism $\theta: G\rightarrow \GL(V)$.
An immediate consequence of Lemma \ref{rat_hom} is that
if $G$ is connected, then $\theta$ is uniquely determined by
any of its restrictions to an open subset
and
any rational homomorphism of algebraic groups $G\to GL(V)$
determines a representation.



Similar to the case of abstract groups, we have the following proposition:
\begin{prop}\cite[Section 4.3]{Xan}(cf. Proposition~\ref{aut_mod}.)\label{alg_aut}
	Suppose that $V$ is a finite-dimensional indecomposable $\fg$-module, where $\fg$ is the Lie algebra of the algebraic group $G$ over $\bK$. 
Then as algebraic groups we have $$\Aut_\fg(V)=\bK^\times\times (1+J)$$ where $J$ is the Jacobson radical of $\End_\fg(V)$. Furthermore, $1+J$ is a connected unipotent algebraic subgroup of $\Aut_\fg(V)$.
\end{prop}

\subsection{Rational and algebraic cohomologies}
\label{s3.2}
Let $H$ be an affine group scheme acting on 
an additive algebraic group $(A,+)$
algebraically by automorphisms. The following easy lemma 
shall be useful in what follows.
\begin{lemma}\label{prim}
	Let $H$ be an irreducible affine group scheme. Then $H$ is primary, i.e., every zero-divisor 
	in $\bK[H]$ lies inside the nilradical.
\end{lemma}
\begin{proof}
	The affinity of $H$ tells us that $\bK[H]=\bK[y_1,\ldots,y_n]/I$ for some $n\geq 1$ and some Hopf ideal $I$. In particular, 
	$I$ has a primary decomposition $I=Q_0\cap\ldots\cap Q_r$ (which we assume to be normal) with associated primes $P_0=\sqrt{I},P_1,\ldots,P_r$. From the perspective of group schemes, this uniquely endows $H$ with a finite collection $p_0,p_1,\ldots,p_r$ of embedded 
        points of $H$,
where $p_i$ is a generic point of
        the irreducible closed subscheme given by $Q_i$. 
        Furthermore, for $i>0$ each $p_i$ is of codimension at least one.
        If $x$ is  a closed point in $H$, then the set $xp_0,xp_1,\ldots,xp_r$ corresponds to the associated primes of another primary decomposition of $I$.
        Hence, by uniqueness, $x$ acts on the set $p_0,p_1,\ldots,p_r$ by permutation. Thus, $H_{red}=\bigcup_{i=1}^{r}(\bigcup_{x {\tiny \mbox{ closed point}}}xp_i)_{red}=\bigcup_{i=1}^{r}(p_i)_{red}$. However, over an algebraically closed field, $H_{red}$ cannot be a finite union proper subvarieties. Hence, $r=0$ and the result follows.
\end{proof}

Define the cochain complex $(C^n_{Rat}(H;A),d)$ to consist of the rational maps
$H^n\rightarrow A$ defined at $(1,1,\ldots, 1)$
with the standard differentials of group cohomology.

A rational function $f$ on $H^n$ is defined
on an open dense subset $U\subseteq H^n$,
thus, $U$ has a non-empty intersection
$U_\alpha = U \cap H^n_\alpha$ 
with each irreducible component
$H^n_\alpha$ of $H^n$.
Since $H^n$ is a group scheme, its irreducible components are connected components
that yields the direct sum decomposition of functions:
$$
\bK [H^n] = \oplus_\alpha \bK [H^n_\alpha]. 
$$
Note that each $H_{\alpha}$ is isomorphic to an irreducible affine group scheme, so we can apply Lemma~\ref{prim}.
Thus, $U_\alpha$ is of the form $U(s_\alpha)$ for
a non-zero-divisor $s_\alpha \in \bK [H^n_\alpha]$
and $f= hs^{-1}$ for some $h \in \bK [H^n]$
and a non-zero-divisor $s \coloneqq (s_\alpha) \in \bK [H^n]$.
Thus, $f\in \bK[H^n]_S$,
the localised ring of functions 
obtained by inverting the set $S$ of all non-zero-divisors.

Writing functions on the algebraic group $A$ as
$\bK[A]=\bK[x_1, \ldots x_m]/I$, a rational $n$-cochain $\mu$
is uniquely determined by an $m$-tuple of rational functions
$(\mu_i)\in \bK[H^n]_S^m$ satisfying the relations of $I$. 
In particular, if each component of $H$ is infinitesimal,
$$
\bK[H^n]_S= \bK[H^n]
\ \ \mbox{ and } \ \
C^n_{Rat}(H;A)=C^n_{Alg}(H;A) \; ,
$$
where, in general,  
$(C^n_{Alg}(H;A),d)$ is the cochain subcomplex if $(C^n_{Rat}(H;A),d)$
that consists of those rational maps $H^n\rightarrow A$
which are, in fact, algebraic.

Let us now concentrate on a connected algebraic group $G$ and
its connected subgroup scheme $L$. 
There is another subcomplex of $(C^n_{Rat}(G;A),d)$ which we are interested in:
we define $(\widetilde{C}^\bullet_{Rat}(G,L;A),d)$ to consist of rational maps $G^n\rightarrow A$
that are trivial on $L^n$ (i.e., everywhere $0$ on $L^n$). As in the case of abstract groups, we define  $(C^\bullet_{Rat}(G,L;A),d)$ by 
$$C^n_{Rat}(G,L;A)=
   \begin{cases}
     \widetilde{C}^n_{Rat}(G,L;A), & \text{if}\ n>0, \\
     A^L, & \text{if}\ n=0 .
     \end{cases}
$$

There is a natural inclusion of cochain complexes $C^\bullet_{Rat}(G,L;A)\to C^\bullet_{Rat}(G;A)$. We can hence define the cochain complex $\widetilde{C}^\bullet_{Rat}(L;A)$ such that $\widetilde{C}^n_{Rat}(L;A)\coloneqq \Coker(C^n_{Rat}(G,L;A)\to C^n_{Rat}(G;A))$ for all $n\geq 0$.

In particular, this gives us the short exact sequence of cochain complexes
$$0\to C^\bullet_{Rat}(G,L;A)\to C^\bullet_{Rat}(G;A)\to \widetilde{C}^\bullet_{Rat}(L;A)\to 0.$$

We define the algebraic complexes $C^\bullet_{Alg}(G,L;A)$ and $\widetilde{C}^\bullet_{Alg}(L;A)$ in the expected way, and once again get a short exact sequence of cochain complexes.
In either case, this allows us to construct the long exact sequence in cohomology (suppressing the `Rat' and `Alg'):
\begin{equation}
\label{Seq_star}
0\to H^1(G,L;A)\to\ldots\rightarrow \widetilde{H}^{n-1}(L;A)\rightarrow H^n(G,L;A)\rightarrow
H^n(G;A)\rightarrow \widetilde{H}^n(L;A)\to\ldots
\end{equation}
Note that $\widetilde{H}^0_{Rat}(L;A)=\widetilde{H}^0_{Alg}(L;A)=0$, hence, 
this exact sequence starts in degree one.

These long exact sequences can be connected, using the maps induced by the inclusions $C^n_{Alg}(G,L;A)\hookrightarrow C^n_{Rat}(G,L;A)$ and $C^n_{Alg}(G;A)\hookrightarrow C^n_{Rat}(G;A)$:
$$
\begin{CD}
\ldots @>>> H^n_{Alg}(G,L;A) @>>> H^n_{Alg}(G;A) @>>> \widetilde{H}^n_{Alg}(L;A) @>>> {H}^{n+1}_{Alg}(G,L;A) @>>> \ldots\\
@. @VVV       @VVV  @VVV   @VVV @.\\
\ldots @>>>H^n_{Rat}(G,L;A) @>>> H^n_{Rat}(G;A) @>>> \widetilde{H}^n_{Rat}(L;A) @>>> {H}^{n+1}_{Rat}(G,L;A) @>>> \ldots.
\end{CD}
$$

Since we identify $C^0_{Alg}(G;A)$ with algebraic maps from the trivial algebraic group to $A$ (and similarly in the other complexes), there is no distinction between rational and algebraic maps. Hence, 
$$H^0_{Rat}(G;A)=H^0_{Alg}(G;A)=H^0_{Rat}(G,L;A)=H^0_{Alg}(G,L;A)=A^G.$$ 

The cocycle condition on $f\in C^1_{Rat}(G;A)$ is precisely the condition considered in Lemma~\ref{rat_cocycle} for a rational map $f:G\rightarrow A$.
Since $G$ is connected, Lemma~\ref{rat_cocycle} tells us the map extends to an algebraic map. Hence, in this case 
$$H^1_{Rat}(G;A)=H^1_{Alg}(G;A)\  
\mbox{ and } \ 
H^1_{Rat}(G,L;A)=H^1_{Alg}(G,L;A).$$

This leads to the following proposition. The first part of it follows from the exact sequence. The second part has a similar proof as Proposition~\ref{cohom_compare}.
\begin{prop}(cf. Proposition~\ref{cohom_compare})
	\label{alg_cohom_compare}
	\begin{enumerate}
		\item If $\widetilde{H}_{Rat}^{1} (L;A)=0$,  then
		$H^1_{Rat}(G,L;A)=H^1_{Rat}(G;A)$.
		\item For $n>0$, if the natural map $Z^{n-1}_{Rat} (G;A)\rightarrow \widetilde{Z}^{n-1}_{Rat} (L;A)$ is surjective, then the natural map
		$H^n_{Rat}(G,L;A)\rightarrow H^n_{Rat}(G;A)$ is injective.
	\end{enumerate}
\end{prop}



The appropriate long exact sequence yields the following.

\begin{cor}
	 $H^2_{Rat}(G,L;A)=0$ if and only if $H^{1}_{Rat}(G;A)\rightarrow \widetilde{H}^{1}_{Rat}(L;A)$ is surjective and $H^{2}_{Rat}(G;A)\rightarrow \widetilde{H}^{2}_{Rat}(L;A)$ is injective.
\end{cor}

When the action is trivial, we can learn more about what these cohomology groups are.

\begin{lemma}\label{zero}
	If $G$ acts trivially on $A$ and $\Hom(L,A)=0$, then $\widetilde{Z}^{1}_{Rat} (L;A)=0$.
\end{lemma}

\begin{proof}
	Let $\mu + C^1_{Rat}(G,L;A)\in \widetilde{Z}^{1}_{Rat} (L;A)$, so $d\mu\in C^2_{Rat}(G,L;A)$. In particular, $d\mu\vert_{L^2}=0$. However, since the action is trivial, $d\mu\vert_{L^2}=0$ if and only if $\mu\vert_{L}$ is a rational homomorphism $L\to A$ if and only if $\mu\vert_L$ is a homomorphism $L\to A$ (since $L$ is connected, by assumption). Since $\Hom(L,A)=0$, we conclude that $\mu + C^1_{Rat}(G,L;A)=0+C^1_{Rat}(G,L;A)$. Hence, $\widetilde{Z}^{1}_{Rat} (L;A)=0$.
\end{proof}

\begin{lemma}\label{H1}
  Let $G$ be a connected algebraic group which acts trivially on a commutative algebraic group $A$. Let $L\leq G$ be a closed connected subgroup scheme. Then $H^1_{Rat}(G;A)=\Hom(G,A)$ and
  $H^1_{Rat}(G,L;A)=\{\mu\in \Hom(G,A)\,\vert\,\mu\vert_{L}\equiv 0\}$. 
\end{lemma}

\begin{proof}	
	Since the $G$-action on $A$ is trivial, the coboundary map $C^0_{Rat}(G;A)\rightarrow C^1_{Rat}(G;A)$ is just the trivial map. Hence, we get that $H^1_{Rat}(G;A)=Z_{Rat}^1(G;A)$,  the rational 1-cocycles of $G$. However, as the action is trivial, rational 1-cocycles of $G$ on $A$ are the same as homomorphisms of algebraic groups $G\rightarrow A$.
	Hence,  $H^1_{Rat}(G;A)=\Hom(G,A)$. 
	
	Essentially the same argument gives $H^1_{Rat}(G,L;A)=\{\mu\in \Hom(G,A)\,\vert\,\mu\vert_{L}\equiv 0\}$.
\end{proof}

Combining Lemma~\ref{H1} with Lemma~\ref{zero} and Proposition~\ref{alg_cohom_compare}(2), we get the following corollary.

\begin{cor}\label{H1+H2}
	Let $G$ be a connected algebraic group acting algebraically (not necessarily trivially) by automorphisms on a commutative algebraic group $A$. Let $L\leq G$ be a connected closed subgroup scheme of $G$ such that the action of $L$ on $A$ is trivial, and $\Hom(L,A)=0$. Then $H^1_{Rat}(G,L;A)= H^1_{Alg}(G;A)$ and $H^2_{Rat}(G,L;A)\rightarrow H^2_{Rat}(G;A)$ is injective.
\end{cor}

The following lemma by van der Kallen \cite[Prop. 2.2]{vdKa} 
is useful in what follows.

\begin{lemma}\label{perfect}
	Let $G$ be a semisimple, simply-connected algebraic group. 
Suppose further that, if $p=2$, the Lie algebra $\fg$ of $G$ does not contain $A_1, B_2$ or $C_l$ ($l\geq3$) as a direct summand. Then $\fg$ is perfect, i.e., $\fg=[\fg,\fg]$.
\end{lemma}

\begin{proof}
	It is enough to prove this result for $G$ simple and simply-connected, with irreducible root system $\Phi$. It is well known that $\fg$ is simple and non-abelian (and so $\fg=[\fg,\fg]$) in the following cases: $p\nmid l+1$ in type $A_l$, $p\neq 2$ in types $B_l,C_l,D_l$, $p\neq 2,3$ in types $E_6,E_7,F_4,G_2$, and $p\neq 2,3,5$ in type $E_8$. It is further known \cite{CKR} that $\fg$ is simple and non-abelian in the following cases: $p=2$ in types $E_6,G_2$, $p=3$ in types $E_7,F_4$, and $p= 2,3,5$ in type $E_8$.
	
	Furthermore, it is known from Table 1 in \cite{Hog} that $\fg=[\fg,\fg]$ in all the remaining cases except for $p=2$ in types $A_1,B_2,C_l$ ($l\geq3$).  
\end{proof}

\begin{lemma}\label{H2}
  Let $G$ be a semisimple, simply-connected algebraic group over an algebraically closed field $\bK$ of characteristic $p$
  which acts trivially on a commutative algebraic group $A$. Suppose further that, if $p=2$, the Lie algebra $\fg$ of $G$ does not contain $A_1, B_2$ or $C_l$ ($l\geq3$) as a direct summand.
  Let $G\1$ be the first Frobenius kernel of $G$. Then $H^2_{Rat}(G,G\1;A)=0$.
\end{lemma}

\begin{proof}
  Let us first show that  $H^2_{Rat}(G;A)=0$.
Let $\mu:G\times G\to A$ be a rational cocycle defined on the open set $U\times U$ with $U^{-1}=U$. We can define a local group structure on the set $A\times G$ by setting 
$$(a,g)(b,h)=(a+b+\mu(g,h),gh)
\ \mbox{ and } \ 
(a,g)^{-1}=(-a-\mu(g,g^{-1}),g^{-1}).$$ 
In the language of Weil \cite{Wei}, $A\times U$ is a group-chunk in the pre-group $A\times G$. 
By Weil's theorem \cite{Wei}, there exists an algebraic group $H$ birationally equivalent to $A\times U$ with $\Phi:A\times U\to \Phi(A\times U)$ 
an isomorphism of algebraic group-chunks  
and $\Phi(A\times U)$ a dense open set in $H$.

 Since $H$ is connected it is generated by $\Phi(A\times U)$. Let $f:A\to H$ be the natural algebraic group homomorphism coming from $A\to A\times U$. This is clearly injective and, since $A$ commutes with each element of $A\times U$, $f(A)\subset Z(H)$. Furthermore, the natural projection $A\times U\to G$ extends to a rational (and so algebraic) homomorphism $\pi:H\to G$, which is surjective as $U$ generates $G$ (since $G$ connected). Finally, it is clear that $f(A)=\ker\pi\cap \Phi(A\times U)$. Hence, $\pi$ descends to a homomorphism $\bar{\pi}:H/f(A)\to G$, whose kernel is discrete (since $\Phi(A\times U)$ is dense in $H$) and, hence, central (as $G$ connected).
 
 In other words, we have a central extension $1\to A\to H\to G\to 1$ of algebraic groups, which corresponds to an algebraic cocycle $\widetilde{\mu}:G\times G\to A$. It is straightforward to see that $\widetilde{\mu}\vert_{U\times U}=\mu\vert_{U\times U}$, and hence $[\mu]$ lies in the image of the natural map $H^2_{Alg}(G;A) \rightarrow H^2_{Rat}(G;A)$. Therefore, the map $H^2_{Alg}(G;A) \rightarrow H^2_{Rat}(G;A)$ is surjective. 

  It suffices to prove that $H^2_{Alg}(G;A)=0$ when $A$ is $\bG_a$ or $\bG_m$ or a finite group:
  the long exact sequence in cohomology reduces the case of arbitrary $A$ to one of these cases.
  It is known that   $H^2_{Alg}(G;\bG_a)=H^2(G;\bK_{triv})=0$ \cite[II.4.11]{Jan}.

  Consider a non-trivial cohomology class in $H^2_{Alg}(G;A)$ when  $A$ is $\bG_m$ or a non-trivial finite group.
  It yields a non-split central extension $1\rightarrow A \rightarrow \widetilde{G} \rightarrow G \rightarrow 1$.
  Pick a non-trivial character $\chi: A \rightarrow \bG_m$. There exists an irreducible representation of $\widetilde{G}$
  with a central character $\chi$. It is an irreducible projective representation of $G$.
  By the original version of
  Steinberg's tensor product theorem \cite{Ste}
  it is linear. Hence, $\chi$ is trivial. This contradiction proves that $H^2_{Alg}(G;A)=0$ for these two particular $A$.
  We have finished the proof that $H^2_{Rat}(G;A)=0$ for an arbitrary $A$.

	Since $G\1$ is a height 1 group scheme, rational homomorphisms of schemes $G\1\rightarrow A$ are fully controlled by the corresponding restricted homomorphisms
	of Lie algebras $\fg\rightarrow \Lie(A)$.
	By Lemma \ref{perfect},
	$\fg=[\fg,\fg]$ and thus all such homomorphism of Lie algebras are trivial. Hence, 
	we can apply Corollary~\ref{H1+H2} to get that 
	$H^2_{Rat}(G,G\1;A)\rightarrow H^2_{Rat}(G;A)$
	is injective, and so $H^2_{Rat}(G,G\1;A)=0$.
	%
	\end{proof}

\subsection{$G$-Stable bricks}
\label{s3.3}
In Chapter~\ref{s2}, we have introduced the notions of weak
$(L,H)$-morphs and $(L,H)$-morphs for abstract groups. In this section, we discuss how these notions apply to algebraic groups and see how they can be used to shed some light on the lifting of $\fg$-modules to $G$-modules.

Suppose that $G,K$ are algebraic groups over $\bK$, where $G$ is connected, and that $L,H$ are closed subgroup schemes of $G,K$ respectively. We say that a rational map $f:G\rightarrow K$ is  {\em a (weak) $(L,H)$-morph of algebraic groups}
if it satisfies the conditions for a (weak) $(L,H)$-morph of abstract groups, where condition (3) is interpreted for only those $x,y,xy\in \dom(f)$.

In analogy with the case of abstract groups, a weak $(L,H)$-morph of algebraic groups is a homomorphism $G\rightarrow N/H$ with a rational lifting $N/H\rightarrow N$ which satisfies an additional condition. It is clear that if $H$ is normal in $K$ then condition (2) is trivially satisfied. We again have that weak $(L,1)$-morphs are just homomorphisms $G\rightarrow K$, and that weak $(1,K)$-morphs are rational maps $G\rightarrow K$ which preserve the identity.

We say that two weak $(L,H)$-morphs of algebraic groups, $f$ and $g$, are equivalent if $f(x)g(x)^{-1}\in H$ for all $x\in \dom(f)\cap\dom(g)$. Given a homomorphism of algebraic groups $\theta:L\rightarrow K$, we denote by $\LH(G,K)$ the quotient by this equivalence relation of the set of weak $(L,H)$-morphs of algebraic groups from $G$ to $K$ which restrict to $\theta$ on $L$.

Suppose that $X$ is a separated algebraic scheme on which $G$ acts rationally on the right (i.e. the action $X\times G\rightarrow X$ is a rational map), $K$ acts algebraically on the left, and the actions commute. Suppose further that $\theta\in X(\bK)$ is such that $\theta^G\subset \,^{K}\theta$, and that there exists a rational section $K/H\rightarrow K$ where $H=\Sta_K(\theta)$ is the scheme-theoretic stabiliser of $\theta$. 

As in the case for abstract groups, this gives us a rational map
$$
f:G\rightarrow K \ \ \mbox{ characterised by } \ \ 
\,^{f(x)}\theta = \theta^{x}
\ \ \mbox{ for all } x\in U\overset{open}{\subset} G.$$

\begin{lemma}\label{alg_1Hhom}
	The map $f$ defined above is a $(1,H)$-morph of algebraic groups.
\end{lemma}

\begin{proof}
	We can think of $f$ as the composition of the following rational maps
	$$G\hookrightarrow \{\theta\}\times G\rightarrow\,^K\theta\rightarrow K/H\rightarrow K.$$
	
	Note that $^K\theta\rightarrow K/H$ is an algebraic map by Demazure-Gabriel \cite[Proposition 3.2.1]{DG}. We then have that the composition is rational since each domain of definition intersects the previous map's image.
	
	The proof that $f(x)f(y)\in f(xy)H$ for $x,y\in G$ with $f(x),f(y)$ and $f(xy)$ defined is exactly the same as in the abstract case, as is the proof that $f(G)\subset N_K(H)$.
	
\end{proof}

Now we fix algebraic (group, subgroup scheme) pairs $(G,L)$ and $(K,H)$ with $H$ soluble and $G$ connected. Denote by $m_G, m_K$ the corresponding multiplication maps, $\Delta_G, \Delta_K$ the diagonal embeddings, and $inv_G, inv_K$ the inverse maps. Let $\theta:L\rightarrow K$ be a homomorphism of algebraic group schemes. Furthermore, choose $\widetilde{H}$ to be an algebraic subgroup of $H$, characteristic in $N=N_K(H)$ such that $A\coloneqq H/\widetilde{H}$ is commutative. We denote the quotient map $H\rightarrow A$ by $\pi$.

We can define an $N$-action on $H$ by conjugation. Note that since $\widetilde{H}$ is characteristic in $N$, so preserved by conjugation, this passes to an algebraic $N$-action on $A$. Hence, we have an algebraic action of $N$ on $A$ which is trivial on $H$ (since $A$ is commutative). This gives us an algebraic $N/H$-action on $A$. For an element $f\in \LH(G,K)$, we get a rational homomorphism $G\rightarrow N/H$ which is, in fact, algebraic by Lemma \ref{rat_hom}. Thus, every element of $\LH(G,K)$ induces an algebraic $G$-action on $A$. This $G$-action respects the multiplication operation of $A$, i.e. it is an algebraic automorphic $G$-action. 




As in the case for abstract groups, we can form something resembling an exact sequence. Let $\rho$ be a rational $G$-action on $A$, and define
$$[L\widetilde{H}]^\theta {\mbox{\rm mo}}(G,N)_\rho\subset [L\widetilde{H}]^\theta {\mbox{\rm mo}}(G,N),\,\,\quad [LH]^\theta {\mbox{\rm mo}}(G,N)_\rho\subset [LH]^\theta {\mbox{\rm mo}}(G,N)$$
as the subsets of weak morphs which induce the action $\rho$.

We get the following theorem.

\begin{theorem}(cf. Theorem \ref{exact_seq})
	\label{alg_exact_seq}
	For a rational $G$-action $\rho$ on $A$ 
	the following statements hold:
	\begin{enumerate}
		\item There is a restriction map 
		$$
		\Res:
		[L\widetilde{H}]^\theta{\mbox{\rm mo}} (G,N)_\rho
		\longrightarrow 
		[LH]^\theta {\mbox{\rm mo}}(G,N)_\rho, \ \ \
		\Res (\langle f\rangle ) = [f]
		$$
		where 
		$\langle f\rangle$ and $[f]$
		are the equivalence classes 
		in 
		$[L\widetilde{H}]^\theta{\mbox{\rm mo}} (G,N)_\rho$
		and
		$[LH]^\theta {\mbox{\rm mo}}(G,N)_\rho$.
		\item The abelian group $Z_{Rat}^1(G,L;(A,\rho))$
		acts freely on the set $[L\widetilde{H}]^\theta{\mbox{\rm mo}} (G,N)_\rho$
		by
		$$
		\gamma \cdot \langle f \rangle \coloneqq \langle \dot{\gamma} f \rangle
		\ \mbox{ where } \ 
		\dot{\gamma}f = m_K\circ (\dot{\gamma}\times f)\circ\Delta_G$$  
		and $\dot{\gamma}: G \xrightarrow{\gamma} A \rightarrow H$
		comes from a rational Rosenlicht section $A\rightarrow H$ (cf. \cite[Theorem 10]{Ros2}) with
		$\dot{\gamma} (1)=1$.
		\item The corestricted restriction map 
		$
		\Res:
		[L\widetilde{H}]^\theta{\mbox{\rm mo}} (G,N)_\rho
		\longrightarrow 
		\mbox{\rm Im}(\Res) 
		$
		is a quotient map by the $Z_{Rat}^1(G,L;(A,\rho))$-action.  
		\item If $H$, $\widetilde{H}$ and $A$ are reduced, two classes $\langle f\rangle, \langle g\rangle\in[L\widetilde{H}]^\theta{\mbox{\rm mo}} (G,N)_\rho$ 
		lie in the same $B_{Rat}^1(G,L;(A,\rho))$-orbit if and only if
		there exist
		$h\in H$, $f^\prime\in \langle f\rangle$, $g^\prime\in \langle g\rangle$
		such that $[f(L),h]\subset\widetilde{H}$ and $f^\prime (x) = h g^\prime (x) h^{-1}$
		for all $x \in G$.
		
		\item
		There is an obstruction map
		$$\Obs:
		[LH]^\theta\mbox{\rm mo}(G,N)_\rho
		\longrightarrow 
		H_{Rat}^2(G,L;(A,\rho)), \ \ \
		\Obs ([f]) = [f^\sharp]
		$$
		where the cocycle $f^\sharp$ is defined by	$$G\times G\xrightarrow{(p_1,p_2,m_K)} G\times G\times G \xrightarrow{(f,f,inv_K f)} K\times K\times K\xrightarrow{m_K} H\xrightarrow{\pi} A$$
		Here, $p_1$ and $p_2$ denote projection to the first and second coordinate respectively.
		
		
		\item The sequence (cf. Sequence~(\ref{exact_seq0}))
$$
[L\widetilde{H}]^\theta{\mbox{\rm mo}} (G,N)_\rho
\longrightarrow 
\LH(G,N)_\rho
\longrightarrow 
H^2_{Rat}(G,L;(A,\rho))
$$
is exact, i.e., the image of $\Res$ is equal to $\Obs^{-1}([0])$.
	\end{enumerate}
\end{theorem}

\begin{proof}
	If $\langle f\rangle =\langle g \rangle$ then the map $$\alpha: G\xrightarrow{(f, inv_K g)} K\times K\xrightarrow{m} K$$ has image in $\widetilde{H}$ and is trivial on $L$. It is rational as it is a composition of rational maps, and the identity is in the domain of definition and image of each map.
	
	We also observe that given an analogous $\alpha:G\rightarrow H$ (i.e. corresponding to $[f]=[g]$) we get $\pi\alpha:G\rightarrow A$. Denoting the Rosenlicht section \cite[Theorem 10]{Ros2} $A\rightarrow H$ by $\tau$, we see that $\tau\pi\alpha=\alpha$ and thus $\dot{(\pi\alpha)}=\alpha$. Note that we may assume that the Rosenlicht section is defined at $0_A$ by composing with a translation if necessary. All the maps here are rational. In particular, $\pi\alpha\in C^1_{Rat}(G,L;(A,\rho))$.
	
	With these observations in mind, the remainder of the proof follows in the same way as in the proof of Theorem \ref{exact_seq} does for abstract groups, doing everything diagrammatically. 
	
\end{proof}

Before going any further, let's consider the following case where we can use this exact sequence directly.
A restricted $\fg$-module $(V,\theta)$ satisfying the condition that $\Aut_\fg(V)=\bK^\times$ is called a {\em brick}. A brick is necessarily an indecomposable $\fg$-module.

\begin{theorem}
	\label{brick}
	Suppose $G$ is a semisimple, simply-connected algebraic group over an algebraically closed field $\bK$ of characteristic $p>0$, with Lie algebra $\fg$. Suppose further that, if $p=2$, $\fg$ does not contain $A_1, B_2$ or $C_l$ ($l\geq3$) as a direct summand. Let $(V,\theta)$ be a finite-dimensional $G$-stable brick. Then there exists a unique $G$-module structure $\Theta$ on $V$ with $\Theta\vert_{G_{1}}=\theta$.  
\end{theorem}

\begin{proof}
	We use Theorem \ref{alg_exact_seq} in the following situation:
	\begin{itemize}
		\item $L=G_{1}$, the first Frobenius kernel of $G$, 
		\item $K=\GL(V)$,
		\item $H=\Aut_\fg(V)=\bK^\times$,
		\item $N=N_K(H)$,
		\item $X=\Hom_{\bK}(\fg,\fg\fl(V))$, a separated algebraic scheme with $\theta\in X(\bK)$.
	\end{itemize}
	
	
	Observe that $G$ acts on $X$ on the right via the adjoint map on the domain and $\GL(V)$ acts on $X$ on the left via conjugation on the image. Furthermore, the actions commute, and the $G$-stability of $V$ gives us that $\theta^G\subset \,^{\mbox{\tiny GL}(V)}\theta$.
	
	Hence, Lemma \ref{alg_1Hhom} gives us a $(1,H)$-morph of algebraic groups, say $f:G\rightarrow \GL(V)$. In particular, it gives a homomorphism of algebraic groups $f:G\to\PGL(V)$, together with a rational lifting $\eta:\PGL(V)\to \GL(V)$. This rational lifting can be defined as follows: fix a basis of $V$ and let $U$ be the open subset of $\PGL(V)$ consisting of all cosets which can be represented by a (unique) matrix $A=(a_{ij})\in\GL(V)$ with $a_{11}=1$. Then define the map $\eta:U\to \GL(V)$ by assigning to each coset this representative.  
	
	Currently $f$ and $\theta$ give the same maps from $G_{1}$ to $N/H$ -- since 
$$
\,^{\theta(x)}\theta(a)(v)=\theta(x)\theta(a)\theta(x^{-1})(v)=\theta(xax^{-1})(v)=\theta^{x}(a)(v)
$$ 
for $x,a\in G_{1}(\bS)$, $v\in V(\bS)$ for any commutative 
$\bK$-algebra $\bS$. Note, however, that the maps $G\1\rightarrow K$ do not necessarily agree. 

	To fix this potential disagreement, we define a rational map $R:G_{1}\rightarrow H=\bK^\times$ by $R(g)=f(g)^{-1}\theta(g)$ for $g\in G_{1}(\bS)$.
	There exists a rational map $\widetilde{R}:G\rightarrow H=\bK^\times$ which restricts to $R$ on $G_{1}$. Indeed, we have $R\in\bK[G_1]$ (as $G_1$ is infinitesimal), so we can lift it to $\widetilde{R}\in\bK[G]$ (since $\bK[L]$ is a quotient of $\bK[G]$). Let $U=G\setminus\widetilde{f}^{-1}(0)$. This is open in $G$, and on $U$ we have that the image of $\widetilde{R}$ lies inside $\bK^\times$, so $\widetilde{R}$ is a rational map $G\to\bK^\times$. If now we define $\widetilde{f}:G\rightarrow \GL(V)$ by $\widetilde{f}(g)=f(g)\widetilde{R}(g)$, we get that $\widetilde{f}$ is a $(G_{1},H)$-morph which restricts to $\theta$ on $G_{1}$, fixing the disagreement.

	Observe that with $\widetilde{H}\coloneqq 1$, we get (in the notation of the Theorem~\ref{alg_exact_seq}) $A=H$ and $G$ acting on $A$ trivially. Hence, the ``exact sequence'' from Theorem \ref{alg_exact_seq} is 
	$$H^1_{Rat}(G,G_{1};\bK^\times)\dashrightarrow [G_{1} 1]^\theta\mbox{\rm mo}(G,N)_1\rightarrow [G_{1}H]^\theta\mbox{\rm mo}(G,N)_1\rightarrow H^2_{Rat}(G,G_{1};\bK^\times)$$
	
	By Lemma \ref{H2}, $H^2_{Rat}(G,G_{1};\bK^\times)=0$. Hence $[\widetilde{f}]\in [G_{1}H]^\theta\mbox{\rm mo}(G,N)_1$ can be lifted to $\widehat{f}\in[G_{1} 1]^\theta\mbox{\rm mo}(G,N)_1$. This means that $\Theta\coloneqq\widehat{f}:G\rightarrow \GL(V)$ is a homomorphism of algebraic groups which restricts to $\theta$ on $G_{1}$. Furthermore, this representation is unique (up to equivalence) if $H^1_{Rat}(G,G_{1};\bK^\times)=0$. 
	
	By Lemma \ref{H1}, $H^1_{Rat}(G,G_{1};\bK^\times)=\{\mu\in \Hom(G;\bK^\times)\,\vert\,\mu\vert_{G_{1}}\equiv 1\}$. Since $G$ is perfect, $H^1_{Rat}(G,G_{1};\bK^\times)=0$ and the extension is unique.
\end{proof}

\subsection{$G$-Stable modules with soluble automorphisms}
\label{s3.4}
We return to the general situation, where $(G,L), (K,H)$ are algebraic (group, subgroup scheme) pairs with $H$ soluble, $G$ connected, and $H$ reduced. However, from now on we suppose that $L$ is a normal subgroup scheme of $G$. We also fix a homomorphism of algebraic groups $\theta:L\rightarrow K$, where the image commutes with $H$, so we are now dealing with $(L,H)$-morphs.
Everything in the previous section can be reformulated in terms of $(L,H)$-morphs without difficulty - the key difference is that the $G$-action on $A$ is now trivial on $L$. Since $H$ is soluble, we can find a subnormal series $H=H_0\rhd H_1\rhd\ldots\rhd H_k=\{1\}$ with commutative quotients $A_j=H_{j-1}/H_j$ and each $H_j$ characteristic in $N=N_K(H)$ and reduced.

Suppose that $f$ is an $(L,H)$-morph of algebraic groups such that $f\vert_L=\theta$. As in the case of abstract groups, we get the following theorem -- it generalises the procedure which we have used for bricks in the previous section.

\begin{theorem}(cf. Theorem \ref{branching})
	\label{alg_branching}
	Given an $(L,H)$-morph of algebraic groups $f=f_0$ with $f\vert_L=\theta$, we obtain any $(L,1)$-morph extending $\theta$ by applying the following procedure. Step $m$ is the following:
	
	\begin{enumerate}
		\item The $(L,H_{m-1})$-morph $f_{m-1}:G\rightarrow N$ such that
		$f_{m-1}|_L=\theta$ determines a rational $G$-action $\rho_m$ on $A_m$.
		\item If $\Obs ([f_{m-1}])\neq 0\in H^2_{Rat} (G, L;(A_m,\rho_m))$, then
		this branch of the process terminates.
		\item If $\Obs ([f_{m-1}])=0\in H^2_{Rat} (G, L;(A_m,\rho_m))$, then
		we choose an $(L,H_{m})$-morph $f_{m}:G\rightarrow N$
		such that $\Res ([f_m])=[f_{m-1}]$.
		\item For each element of $H_{Rat}^1 (G, L;(A_m,\rho_m))$ we choose
		a different $f_m$ branching the process. (The choices
		different by an element of $B_{Rat}^1 (G, L;(A_m,\rho_m))$
		are conjugate by an element of $H$.)
		\item We change $m$ to $m+1$ and go to step (1).
	\end{enumerate}
	
	An $(L,1)$-morph which restricts to $\theta$ on $L$ is equivalent to $f_k$ for one of the non-terminated branches. Two $(L,1)$-morphs $f, g$ come from different branches if and only if 
	there is no $h\in H$ such that $f(x)=h g(x)h^{-1}$ for all $x\in G$.
\end{theorem}

We get the following corollaries, similarly to Section~\ref{s2.4}:

\begin{cor}
	Suppose $H_{Rat}^2(G,L;(A_m,\rho_m))=0$ for all $m$ for one of the branches. Then this branch does not terminate and there is a homomorphism $f:G\rightarrow K$ which restricts to $\theta$ on $L$.
\end{cor}

\begin{cor}
	Suppose $H_{Rat}^1(G,L;(A_m,\rho_m))=0$ for all $m$ for one of the non-terminating branches. Then this branch is the only branch. Moreover, if a homomorphism of algebraic groups $f:G\rightarrow K$ restricting to $\theta$ exists, then it is unique up to conjugation by an element of $H$.
\end{cor}

\begin{cor}
	Suppose $H_{Rat}^1(G,L;(A_k, \rho_k))\neq 0$ for one of the non-terminating branches. Then there exist algebraic homomorphisms $G\rightarrow K$ which are not conjugate by an element of $H$.
\end{cor}

We apply this theorem (and these corollaries) in the following case - a generalisation of the case from the previous section:

\begin{itemize}
	\item $G$ -- connected algebraic group over $\bK$ with Lie algebra $\fg$,
	\item $L=G_{1}$,
	\item $K=\GL(V)$, where $(V,\theta)$ is a finite-dimensional $G$-stable indecomposable $\fg$-module,
	\item $H=\Aut_\fg(V)$, 
	\item $X=\Hom_{\bK}(\fg,\fg\fl(V))$, a separated algebraic scheme with $\theta\in X(\bK)$.
\end{itemize}

Applying exactly the same argument as in Theorem~\ref{brick}, we only start to encounter problems when trying to extend the rational map $R:G_{1}\rightarrow H$ to a rational map on the whole of $G$. This can be fixed without much difficulty.





As a variety, we have that $H=\bK^\times\times \bK^n\subset \bK^{n+1}$ for some $n$ [Proposition \ref{alg_aut}]. Hence, we get $R=(R_0,R_1,\ldots,R_n)$ where $R_i\in \bK[G_{1}]$ for $i=0,1,\ldots,n$. We can then lift each of these to elements of $\bK[G]$, so we obtain $\widetilde{R}=(\widetilde{R_0},\widetilde{R_1}\ldots,\widetilde{R_n}):G\rightarrow \bK^{n+1}$. We would like the image to lie in $H$. 
Thus,  we define $U=G\setminus R_0^{-1}(0)$. This is an open set in $G$, so we can view $\widetilde{R}$ as a rational map from $G$ to $\bK^\times\times \bK^n=H$ which is defined on $U$, and restricts to $R$ on $G_{1}$.

Now we can define $\widetilde{f}:G\rightarrow \GL(V)$ as $\widetilde{f}(g)=f(g)\widetilde{R}(g)$.
This is a $(G_{1},H)$-morph of algebraic groups, which restricts to $\theta$ on $G_{1}$. Hence, we are in the situation of Theorem \ref{alg_branching}. Observe that $\theta:G_{1}\rightarrow \GL(V)$ extends to a homomorphism of algebraic groups $\Theta:G\rightarrow \GL(V)$ if and only if there exists a $(G_{1},1)$-morph of algebraic groups extending $\theta$. In particular, the corollaries to Theorem \ref{alg_branching} can be used to determine the existence and uniqueness of a $G$-module structure on $V$.

\begin{cor}{\rm (Existence Test)}
	Suppose that $G$ is a connected algebraic group over $\bK$ with Lie algebra $\fg$, and that $V$ is an indecomposable $G$-stable finite-dimensional $\fg$-module. Then there exists a $G$-action on $V$, which respects the $\fg$-module structure, if and only if there is a branch (in the terminology of Theorem \ref{alg_branching}) which does not terminate; for instance, a branch such that $H^2_{Rat}(G,G_{1};(A_m,\rho_m))=0$ for all $(A_m,\rho_m)$ on that branch.
\end{cor}

\begin{cor}\label{alg_Uniq} {\rm (Uniqueness Test)}
	Suppose that $G$ is a connected algebraic group over $\bK$ with Lie algebra $\fg$, and that $V$ is an indecomposable $G$-stable finite-dimensional $\fg$-module. Suppose further that there exists a $G$-action on $V$ 
which extends the $\fg$-module structure. This $G$-action is unique (up to isomorphism) if and only if there is a branch (in the terminology of Theorem \ref{alg_branching}) such that $H^1_{Rat}(G,G_{1};(A_m,\rho_m))=0$ for all $(A_m,\rho_m)$ on that branch.
\end{cor}

Observe that combining Corollary~\ref{alg_Uniq} with Corollary \ref{H1+H2} for the $N$-stable subnormal series $H_m=1+J^m$, $m\geq 1$, we get a similar result to Proposition 4.3.1 in \cite{Xan}.
 
\subsection{Comparison with $C^\bullet_{Rat}(G/L;A)$}
\label{s3.5}

Let us now mimic the approach we took in Section \ref{s2.5} and examine how our cochain complex $(C^\bullet_{Rat}(G,L;A),d)$ compares with the complex $(C^\bullet_{Rat}(G/L;A),d)$ on the level of cohomology. We use the notation of Section~\ref{s3.3}. As with our discussion in Section~\ref{s2.5} we have to assume that $L$ acts trivially on $A$ for this discussion to be meaningful -- a condition which holds in the examples considered.

Similar to the case for abstract groups, we have the following proposition.

\begin{prop}\label{H1Map2}
	Under the aforementioned conditions
	we have isomorphisms of groups
	\newline 
	$H^0_{Alg}(G,L;A) \cong H^0_{Alg} (G/L;A)$
	and
	$H^1_{Alg} (G,L;A)\cong H^1_{Alg} (G/L;A)$.
\end{prop}

\begin{proof}
	Making use of the universal property of the quotient for algebraic groups, the proof follows word-for-word as in Proposition~\ref{H1Map}.	
\end{proof}

Recalling the observation that there is no distinction between $H^i_{Alg}$ and $H^i_{Rat}$ for $i=0,1$ this tells us that $H^0_{Rat}(G,L;A) \cong H^0_{Alg} (G/L;A)$
and
$H^1_{Rat} (G,L;A)\cong H^1_{Alg} (G/L;A)$ in these circumstances. 

The universal property of the quotient for algebraic groups further yields an analogue of Proposition \ref{H2inj}.

\begin{prop}
  The maps $\Inf_{Alg}:H^2_{Alg}(G/L;A)\to H^2_{Alg}(G,L;A)$ and \newline
  $\Inf_{Rat}:H^2_{Rat}(G/L;A)\to H^2_{Rat}(G,L;A)$ are injective.
\end{prop}

\begin{proof}
	The proof follows as in Proposition \ref{H2inj}.
\end{proof}

In the case of abstract groups, Section~\ref{s2.5} shows that by making careful choices of $(L,H)$-morphs in Theorem~\ref{branching} the image of the obstruction maps $\Obs:
\LH(G,N)_{\rho_i}
\longrightarrow 
H^2(G,L;(A_i,\rho_i))$
always lies inside $H^2(G/L;(A_i,\rho_i))\hookrightarrow H^2(G,L;(A_i,\rho_i))$. As such, it is possible to reinterpret Theorem~\ref{branching} using the complex $(C^\bullet(G/L;A),d)$ instead of $(C^\bullet(G,L;A),d)$ at all points. This conclusion for abstract groups, however,
relies on the observation that it is always possible to assume that the $(L,H)$-morphs being considered are normalised. When translating the results to the case of algebraic groups it is far from clear that the analogues of Lemma~\ref{Norm} and Corollary~\ref{H2Map} hold.

{\bf Question:} Can the $(L,H)$-morphs considered in Sections~\ref{s3.3} and \ref{s3.4} be chosen to be normalised?

\end{document}